\documentclass[a4paper,11pt]{amsart}

\usepackage{amsmath}
\usepackage{amssymb,amsbsy,amsmath,amsfonts,amssymb,amscd}
\usepackage{latexsym}
\usepackage{amsthm}
\usepackage{graphics}
\usepackage{color}
\usepackage{tikz}
\usetikzlibrary{arrows}
\usepackage{tikz-cd}
\usepackage{pdfpages} 
\usepackage{longtable}
\usepackage{xy}
\usepackage{chngcntr}
\usepackage{array}
\usepackage{color}

\input xy
\xyoption{all}

\newcommand\la{\lambda}
\newcommand\Lam{\Lambda}
\newcommand\al{\alpha}

\newcommand\fie{\varphi}

\newcommand\eps{\epsilon}

\DeclareMathOperator{\Fix}{Fix}

\DeclareMathOperator{\Gr}{Gr}

\DeclareMathOperator{\Cl}{Cl}

\newcommand{\CC}{\ensuremath{\mathbb{C}}}
\newcommand{\RR}{\ensuremath{\mathbb{R}}}
\newcommand{\ZZ}{\ensuremath{\mathbb{Z}}}
\newcommand{\QQ}{\ensuremath{\mathbb{Q}}}

\newcommand{\NN}{\ensuremath{\mathbb{N}}}

\newcommand{\ci}{\ensuremath{\iota}}
\newcommand{\cg}[1]{\ensuremath{\ZZ_#1}} 
\newcommand{\cgb}[1]{\ensuremath{\ZZ_#1}} 
\newcommand{\cond}[2]{\ensuremath{(\ast_{#2})}}

\def\eea{\end{eqnarray*}}
\def\bea{\begin{eqnarray*}}

	\DeclareMathOperator{\Aut}{Aut}
	\DeclareMathOperator{\End}{End}

	\newcommand\dual{\mathrel{\raise3pt\hbox{$\underline{\mathrm{\thinspace d
						\thinspace}}$}}}
	\newcommand\qe{\ifhmode\unskip\nobreak\fi\quad $\Box$}       

	\def\BOX{\hfill\lower.5\baselineskip\hbox{$\Box$}}

	\newtheorem{theorem}{Theorem}
	
	\newtheorem{remark}[theorem]{Remark}
	\newenvironment{rem}{\begin{remark}\rm}{\end{remark}}
	\newtheorem{question}[theorem]{Question}
	\newtheorem{prop}[theorem]{Proposition}
	\newtheorem{cor}[theorem]{Corollary}
	\newtheorem{lemma}[theorem]{Lemma}
	\newtheorem{example}[theorem]{Example}

	\numberwithin{theorem}{section}

	\theoremstyle{definition}
	\newtheorem{defin}[theorem]{Definition}

	\newcommand{\torus}{\ensuremath{T}}

	\DeclareMathOperator{\GL}{GL}
	\DeclareMathOperator{\Bihol}{Bihol}
	
	\newcommand{\transl}{\ensuremath{T_r}}
	\newcommand{\AltForm}{\ensuremath{E}}
	
	\DeclareMathOperator{\id}{id}
	
	\usepackage{hyperref}
	
	\makeatletter
	\def\tagform@#1{\maketag@@@{\ignorespaces#1\unskip\@@italiccorr}}
	\makeatother
	
	\newcolumntype{H}{@{}>{\lrbox0}l<{\endlrbox}} 

\begin{document}
		\title[Classification of Bagnera-de Franchis varieties]{Classification of Bagnera-de Franchis Varieties in Small Dimensions}
		\author{Andreas Demleitner}
		\address {Lehrstuhl Mathematik VIII\\
			Mathematisches Institut der Universit\"at Bayreuth, NW II\\
			Universit\"atsstr. 30\\
			D-95447 Bayreuth \\
			Germany}
		\email{andreas.demleitner@uni-bayreuth.de}

		\begin{abstract}
			A \textit{Bagnera-de Franchis variety} $X = A/G$ is the quotient of an abelian variety $A$ by a free action of a finite cyclic group $G \subset \Bihol(A)$, which does not contain only translations. Constructing explicit polarizations and using a method introduced by F. Catanese, we classify split Bagnera-de Franchis varieties up to complex conjugation in dimensions $\leq 4$. \par \hfill \par 
		\end{abstract}
		
		\maketitle
		
		\tableofcontents

		\setcounter{section}{-1}
		
		\section{Introduction}
		\setcounter{page}{1}  
		This work studies free group actions of finite groups $G$ on abelian varieties $A$ and the corresponding quotients. Here, the group $G$ is a group of affine transformations of $A$, but not a subgroup of the group of translations (else, the quotient would be again an abelian variety). A quotient of an abelian variety by such a group $G$ is called a \textit{generalized hyperelliptic variety}. More generally, one defines a \textit{generalized hyperelliptic manifold} to be the quotient of a complex torus by a group $G$ as above. \\
		The study of these dates back to the beginning of the 20th century, when Bagnera and de Franchis as well as Enriques and Severi published their seminal works \cite{BdF} and \cite{Enr-Sev}, respectively. In the surface case, the classification result of Bagnera and de Franchis shows that there are no non-projective hyperelliptic manifolds. Since then, several authors have studied hyperelliptic manifolds, as well as related areas that contributed a lot to today's understanding of this topic. To name only a few works: \cite{Uchida-Yoshihara}, \cite{Fujiki}, \cite{BGL}, \cite{Catanese-Ciliberto}.
		In 2001, Lange (\cite{Lange}) gave a method to classify BdF-varieties up to dimension $4$, using heavily the tables of linear automorphisms of abelian surfaces and threefolds (loc. cit), although he omitted some calculations in his work. It does not seem that this method can be used for the classification in dimension $> 4$ (because tables of linear automorphisms are - as far as we know - currently only known up to dimension $3$). Instead, Catanese \cite{Fabrizio} introduced a method for the classification based on elementary linear algebra and number theory which will be explained and used in this paper for the classification in higher dimensions. \par \hfill \par
		
		Let us explain how this work is organized. The first chapter mainly recalls some basic facts we will need and establishes some elementary results concerning combinatorics of automorphisms of complex tori. In section \ref{bdf-chapter} we introduce Bagnera-de Franchis varieties as quotients of an abelian variety $A$ by a free action of a finite cyclic group $G$ which is not a subgroup of the group of translations and state a characterization for them: a BdF-variety $X=A/G$ splits as $A = (B_1 \times B_2)/(G \times T_r)$, where $T_r$ is a finite group of translations, such that suitable properties are satisfied (cf. Theorem \ref{charac}). Here, $G$ acts on $B_1$ by translation and linearly on $B_2$. \\
		In Chapter 2 we follow \cite{Fabrizio} and introduce the \textit{Hodge type} of a $G$-Hodge decomposition, an invariant attached to a faithful representation $G \to \GL(\Lam)$, where $\Lam$ is a free abelian group of even rank. \\
		Catanese's method (loc. cit.) for the classification of BdF-varieties will be discussed in Chapter 3. We will assume here that the lattice $\Lam_2$ of $B_2$ is a module over a direct sum of cyclotomic rings (in this case we call $X$ \textit{split}). This yields a decomposition of our abelian variety $B_2$ into $G$-invariant abelian subvarieties $B_{2,k}$, on which $G$ acts with eigenvalues of order $k$. We go on classifying complex tori which admit a linear automorphism acting only with eigenvalues of order $k$ to be able to list all possible decompositions for $B_2$. In Chapter 4 and Chapter 5, we put all pieces together (such that the conditions in the characterization of BdF-varieties are satisfied) and obtain the following classification result:
		
		\begin{theorem}\label{class-result} The following classification results hold.
			\begin{enumerate}
				\item There are no BdF-curves.
				\item Families of split BdF-varieties $X$ of dimension $\leq 4$ are fully classified up to complex conjugation in \textsc{Tables 5-7}.
				\item Families of complex tori of dimension $\leq 5$, which admit a linear automorphism of order $m := |G|$ whose eigenvalues are only primitive $m$-th roots of unity are fully classified up to complex conjugation in sections \ref{surface-case} to \ref{5dim}. 
				\item Each family of complex tori of dimension $\leq 5$ as in iii) contains an abelian variety.
			\end{enumerate}
							Moreover, except possibly for the cases listed in \textsc{Table 7}, every family listed in iii) contains a principally polarized abelian variety.
		\end{theorem}
		
		The one-dimensional case i) is an easy consequence of the Riemann-Hurwitz formula, while the classification result for two-dimensional BdF-varieties is exactly the classification result of Bagnera-de Franchis, Enriques-Severi (\cite{BdF}, \cite{Enr-Sev}). The threefold case was treated by Lange (\cite{Lange}). However, the result ii) in $\dim(X) = 4$ is new, as well as iii) and iv) are (as far as we know). \\
			The problem we face during the classification of BdF-varieties is that we do not know whether the classified families of complex tori in iii) really contain abelian varieties. This question will be dealt with in the last chapter: we find explicit polarizations for these, which turn out to be principal in most cases. We also investigate the problem of projectivity from another point of view, explaining how the following result by T. Ekedahl (for a detailed proof, see \cite{Demleitner}) applies to our situation.
		
		\begin{theorem} \label{Ekedahl}
			Let $(\torus,G)$ be a rigid group action of a finite group $G$ on a complex torus $\torus$. Then $\torus/G$ is projective (or equivalently, $T = A$ is an abelian variety).
		\end{theorem}
		
		\textbf{Notation}: We fix the following notation throughout the whole work. We will work over the field $\CC$ of complex numbers. By an \textit{abelian variety}, we will therefore mean a complex abelian variety. The notion of a \textit{ring} will always mean a commutative ring with unit element. The set of natural numbers $\NN$ will denote the set of all non-negative integers. The dual space of a vector space $V$ is denoted $V^\vee$. \par \hfill \par
		
		\thanks{\textbf{Acknowledgements:} The author cordially thanks Prof. Fabrizio Catanese for a lot of useful advice regarding this topic. This work was generously supported by the ERC Advanced Grant,  n. 340258, 'TADMICAMT'.}
		
		\section{Preliminaries}
		
		In this section, we recall some basic facts which we will need in the sequel. Let $\torus = V/\Lam$ be a complex torus. It is well-known (cf. for instance \cite{Cpl-Ab-Var} ) that $\torus$ is an abelian variety if and only if there is an alternating $\ZZ$-bilinear form $\AltForm$ on $\Lam$ such that the associated $\RR$-bilinear form $H \colon V \times V \to \RR$ given by $H(v,w) = \AltForm(\ci v,w)+ \ci \AltForm(v,w)$ is Hermitian and positive definite. These conditions are publicly known as the two \textit{Riemann Bilinear Relations}. The Riemann Bilinear Relations can also be expressed in the following way. The form $\AltForm$ extends $\CC$-linearly to a form $\AltForm$ on $\Lam \otimes_\ZZ \CC = V \oplus \overline{V}$. We have
		\begin{align*}
		\AltForm \in (V^\vee \otimes V^\vee) \oplus (V^\vee \otimes \overline{V}^\vee) \oplus (\overline{V}^\vee \otimes V^\vee) \oplus (\overline{V}^\vee \otimes \overline{V}^\vee).
		\end{align*}
		Hence, $\AltForm$ splits as a sum $\AltForm = \AltForm_1 + H_1 + H_2 + \AltForm_2$ (where $\AltForm_1$ is in the first direct summand, $H_1$ is in the second one, and so on). Now we have (cf. \cite[p. 327]{Griffiths-Harris}):
		
		\begin{prop} The following statements hold:
			\begin{enumerate}
				\item The first Riemann Bilinear Relation holds if and only if $\AltForm_1 = 0$.
				\item The second Riemann Relation holds if and only if the first Riemann Relation holds and the form $H_1(-\ci z,\overline{w})$ is positive definite.
			\end{enumerate}
		\end{prop}
		
		\subsection{Combinatorics of Automorphisms of Complex Tori}
		In this section we develop combinatorical and group-theoretical restrictions concerning automorphisms of complex tori. These results will be very useful to determine all possible classes of BdF-varieties (cf. Chapters 3 and 4). \par \hfill \par
		Let $\torus = V/\Lam$ be a complex torus of dimension $n$. Let $\al \in \Aut(\torus)$ and $\rho \colon \Aut(\torus) \to \GL(V)$ be the complex representation. This yields a representation $\rho' \colon \Aut(\torus) \to \GL(\Lam)$, called the rational representation. An easy observation is the following lemma (cf. \cite{Cpl-Ab-Var}).
		
		\begin{lemma} 
			The representations $\rho' \otimes 1$ and $\rho \oplus \overline{\rho}$ are equivalent.
		\end{lemma}
		
		Before we start discussing combinatorial restrictions on automorphisms of complex tori, let us first fix some notation. We define $\Aut(\torus)$ to be the group of all (not necessarily linear) automorphisms of $\torus$ and $\Aut(\torus,0)$ to be the subgroup of all linear automorphisms of $\torus$, and analogously $\End(\torus,0)$. For $\al \in \Aut(\torus,0)$, we define $\Fix(\al)$ as the set consisting of all $a \in \torus$ which are fixed under $\al$. 
		
		The following result can be found, together with its proof, in \cite{BGL}; we give a different, more elementary proof here.
		
		\begin{prop}\label{p-pow} Let $\al \in \Aut(\torus)$ be a linear automorphism of order $m = p^k$ for a prime $p$ and assume that $\Lam$ is a free $\ZZ[\zeta_m]$-module. Suppose furthermore that the eigenvalues of $\rho(\al)$ are all primitive $m$-th roots of unity. Then $\Fix(\al) \cong (\ZZ/p\ZZ)^{2\cdot \dim(\torus) /\fie(m)}$. Here, $\fie$ denotes the Euler totient function.
		\end{prop}
		
		\begin{proof} 
			Denote by $x$ the class of $X$ in $\ZZ[\zeta_m] = \ZZ[X]/(\phi_m(X))$ and note that $\al$ acts as multiplication with $x$. We know that $v \in \Fix(\al)$ iff $v \in \Lam \otimes_\ZZ \QQ$ and $(x-1)\cdot v \in \Lam$. It is well known that the $p^k$-th cyclotomic polynomial is $\phi_p(x^{p^{k-1}}) = \sum_{i=0}^{p-1} x^{p^{k-1}\cdot i}$. Since $\Lam \cong \ZZ[\zeta_m]^{2\cdot \dim(A) /\fie(m)}$, it suffices to show that $\{w \in \ZZ[\zeta_m] \otimes_\ZZ \QQ \, | \, (x-1)\cdot w \in \ZZ[\zeta_m]\} \cong \ZZ/p\ZZ$. \\
			For this, we abbreviate $l = \fie(p^k) = \deg(\phi_{p^k})$ and we write $w = a_0 + a_1x + ... + a_{l-1}x^{l-1}$ with $a_i \in \QQ$ for all $i$. Then a straightforward calculation shows the claim:
			\begin{align*}
			(x-1)w= -\sum_{i=0}^{l-1} a_ix^i + \sum_{i=1}^{l-1} a_{i-1}x^i - \sum_{i=0}^{p-2} a_{l-1} x^{p^{k-1} \cdot i},
			\end{align*}
			where we have used the equality $x^l = -\sum_{i=0}^{p-2}x^{p^{k-1}\cdot i}$. \\
			The condition is now that all the coefficients of powers of $x$ must be integers, so the sum of the coefficients has to be an integer. Adding the coefficients, we find that $-p \cdot a_{l-1}$ has to be an integer. This shows the assertion, as all $a_i$ are congruent to $a_{l-1}$ modulo $\ZZ$.
		\end{proof}
		
		\begin{rem}\label{torsionpoints} Let $\torus \cong_{\text{top.}} (S^1)^{2n}$ be a torus of real dimension $2n$. Then the preceding homeomorphism gives $\torus[m] \cong (\ZZ/m\ZZ)^{2n}$. \end{rem}
		
		\begin{rem} \label{el-of-order-m}
			It follows from standard combinatorics that the number of elements of order exactly $m$ in $(\ZZ/m\ZZ)^{r}$ is precisely
			\begin{align*}
			\sum_{k=0}^{r-1} \fie(m)\cdot(m-\fie(m))^k\cdot m^{r-k-1}.
			\end{align*}
		\end{rem}
		
		\subsection{An Introduction to Bagnera-de Franchis Varieties} \label{bdf-chapter}
		
		We start this section with a few preliminaries on generalized hyperelliptic varieties.
		\begin{defin} \label{hyper}
			A \textit{generalized hyperelliptic variety} (resp. \textit{generalized hyperelliptic manifold}) $X = A/G$ of dimension $n$ is the quotient of an abelian variety (resp. a complex torus) $A$ by a finite group $G$ of affine transformations of $A$, such that $G$ acts freely and is not a subgroup of the group of translations. If $G$ is cyclic, $X$ is called a \textit{Bagnera-de Franchis variety} (resp. Bagnera-de Franchis manifold); for short: BdF-variety (BdF-manifold).\end{defin}
		
		\begin{rem} \label{trans}
			Let $X = A/G$ be a generalized hyperelliptic variety. By virtue of our assumption that, in the above situation, $G$ is not a subgroup of the group of translations, we can assume without loss of generality that $G$ does not contain any translation: \\
			Let $G_T$ be the subgroup of translations in $G$, which is a normal subgroup. Hence, we get an abelian variety $A' = A/G_T$ and a group $G' = G/G_T$ without translations such that $A'/G' \cong A/G$. \par \hfill \par
			Also note that there are no hyperelliptic varieties of dimension $1$ in the above sense, since the Riemann-Hurwitz-formula holds.
		\end{rem}
		
		\begin{rem}\label{eigenvalue} Let $X = A/G$ be a generalized hyperelliptic variety, and $A = V/\Lam$. Write $\id \neq g \in G$ as $g(x) = \al x + b$, where $\al \in \GL(V)$ and $b \in V$. Then the property of $g$ acting freely on $A$ is equivalent to the fact that no pair of $(x,\la) \in V \times \Lam$ solves the equation
			\begin{equation}
			(\al - \id)x = \la - b.
			\end{equation}
			Hence, the freeness of the action of $g \in G$ implies that $\al$ has the eigenvalue $1$. \end{rem}

		\begin{defin}
			A BdF-variety (resp. BdF-manifold) $X = A/G$ is said to be of \textit{product type}, if $A = B_1 \times B_2$ is a product of abelian varieties (resp. complex tori) and $G \cong \ZZ/m\ZZ$ is generated by an automorphism $g$ acting as $g(a_1,a_2) = (a_1 + b', \al'a_2)$, where $b' \in B_1$ is an element of order $m$ and $\al' \in \Aut(B_2)$ is a linear automorphism of order $m$, not admitting the eigenvalue $1$.
		\end{defin}
		
		We have the following characterization of BdF-varieties:
		
		\begin{theorem} \label{charac}
			The following two statements are equivalent:
			\begin{enumerate}
				\item $X = A/G$ is a BdF-variety.
				\item $X$ is the quotient of a BdF-variety $(B_1 \times B_2)/G$ of product type by a finite group $\transl$ of translations such that the following conditions hold.
				\begin{enumerate}
					\item $\transl$ is the graph of an isomorphism $T_{r,1} \to T_{r,2}$, where $T_{r,i}$ is a finite group of translations of $B_i$.
					\item If $g \in G$ is given by $g(a_1,a_2) = (a_1 + b', \al'a_2)$, then $(\al'-\id)T_{r,2} = \{0\}$. \label{2nd}
					\item If $g$ is as above of order $m$, then $\langle b' \rangle \cap T_{r,1} = \{ 0 \}$. \label{3rd}
				\end{enumerate}
			\end{enumerate}
			In particular, we may write $X = (B_1 \times B_2)/(G \times \transl)$.
		\end{theorem}
		
		\begin{proof}
			See for instance \cite[Proposition 15]{Fabrizio}. 
		\end{proof}
		
		\section{Actions of Finite Groups on Complex Tori}
		
		We give the most important definitions for the classification of BdF-varieties.
		
		\begin{defin} Let $\Lam$ be a free abelian group of even rank and $G \to \GL(\Lam)$ be a faithful representation of a finite group $G$. A \textit{$G$-Hodge decomposition} is a decomposition
			\begin{align*}
			\Lam \otimes_\ZZ \CC = H^{1,0} \oplus H^{0,1}, \, \, \, \overline{H^{1,0}} = H^{0,1},
			\end{align*}
			into $G$-invariant linear subspaces.
		\end{defin}

		Splitting $\Lam \otimes_\ZZ \CC$ using the canonical decomposition, we can write 
		\begin{align*}
		\Lam \otimes_\ZZ \CC = \bigoplus_\chi U_\chi,
		\end{align*}
		where the sum runs over all characters belonging to irreducible representations of $G$. Thus, $U_\chi = W_\chi \otimes M_\chi$. Here, $W_\chi$ is the corresponding irreducible representation and $M_\chi$ is a trivial representation. Write
		\begin{align*}
		V := H^{1,0} = \bigoplus_\chi V_\chi
		\end{align*}
		with $V_\chi = W_\chi \otimes M_\chi^{1,0}$. We define
		
		\begin{defin} The \textit{Hodge type} of a $G$-Hodge decomposition is the collection of the dimensions $\nu(\chi) = \dim_{\CC} M_\chi^{1,0}$. Here, $\chi$ runs over all non-real characters.
		\end{defin}
		
		\begin{rem}
			\noindent \begin{enumerate}
				\item Note that for a non-real irreducible character $\chi$, one has $\nu(\chi) + \nu(\overline{\chi}) = \dim_\CC M_\chi$.
				\item All $G$-Hodge decompositions of a fixed Hodge type are parametrized in an open set of a product of Grassmanians: for a real irreducible character $\chi$, one simply chooses a $\frac12\dim(M_\chi)$-dimensional subspace of $M_\chi$, and for a non-real irreducible character, one chooses a $\nu(\chi)$-dimensional subspace of $M_\chi$. Then the condition is that $M_\chi^{1,0}$ and $M_\chi^{0,1} = \overline{M_{\overline{\chi}}^{1,0}}$ do not intersect. 
			\end{enumerate}
		\end{rem}
		
		Let $\rho \colon G \to \GL(V)$ be the linear representation which sends $g \in G$ to its linear part. Denote by $G^\vee$ the group of characters of $G$, i.e., the group of group homomorphisms from $G$ to $\CC^*$. For simplicity, we write $\rho_g$ instead of $\rho(g)$. Since $G$ was assumed to be abelian, all the irreducible representations of $G$ have degree $1$. The $\chi$-eigenspace of $G$ is denoted $V_\chi$, i.e., 
		\begin{align*} V_\chi = \{ v \in V \, | \, \rho_g(v) = \chi(g)\cdot v, \, \forall g \in G\}.
		\end{align*}  \label{unity_eigenspaces}
		Thus, we can split $V = \bigoplus_{\chi} V_\chi$. Denote by $M$ the set of all characters such that $V_\chi \neq \{0\}$. We then have $V = \bigoplus_{\chi \in M} V_\chi$. \\
		Of course, we want to apply these considerations to generalized hyperelliptic varieties. In particular, we will have $G$ satisfying the following two properties:
		\begin{itemize}
			\item $G$ acts freely.
			\item $G$ contains no translations, i.e., $\rho$ is faithful.
		\end{itemize}
		
		We have the following elementary result if $G$ is even cyclic:
		
		\begin{lemma} Let $\colon G \to \GL(\Lam)$ be a representation of a finite cyclic group $G = \langle g \rangle$, $\Lam$ a free abelian group of rank $2n$ and $\Lam \otimes_\ZZ \CC = \bigoplus_\chi U_\chi$ be the canonical decomposition. If two characters $\chi$ and $\chi'$ have the same order, the spaces $U_\chi$ and $U_{\chi'}$ have the same dimension.
		\end{lemma}
		
		\begin{proof}
			We have $G \to \GL(\Lam) \hookrightarrow \GL(\Lam \otimes_\ZZ \CC)$. Denote by $\rho^0$ the representation $G \to \GL(\Lam \otimes_\ZZ \CC)$.
			Let $\chi$ and $\chi'$ be two characters of the same order $k$. Since $\chi$ and $\chi'$ have order $k$ and the characteristic polynomial $f_g$ of $\rho^0_g$ has integral coefficients, a power of the $k$-th cyclotomic polynomial divides $f_g$. This yields that the multiplicities of $\chi(g)$ and $\chi'(g)$ as zeros of $f_g$ are equal. We are now finished, because $\rho_g^0$ is diagonalizable.
		\end{proof}
		
		\section{Primary BdF-Varieties}
		
		In this section, we want to classify split BdF-varieties in small dimensions. We use Theorem \ref{charac}, where we saw that a BdF-variety is the quotient of a BdF-variety of product type by a finite group of translations such that the conditions in the quoted theorem are satisfied. We follow \cite{Fabrizio} for the presentation of the method used for the classification.
		
		\begin{defin}
			\noindent \begin{enumerate}
				\item Assume that a cyclic group $G$ generated by a linear automorphism $g$ of order $m$ acts on an abelian variety (resp. a complex torus) $A$. Then $(A,G)$ is called \textit{primary}, if a generator $g$ of $G$ has only primitive $m$-th roots of unity as eigenvalues. By abuse of notation, we will call an abelian variety (resp. a complex torus) $A$ primary if the group $G$ is clear from the context.
				\item A BdF-variety (resp. BdF-manifold) $(B_1 \times B_2)/G$ of product type is said to be \textit{primary}, if the abelian variety $B_2$ (resp. the complex torus $B_2$) is primary.
			\end{enumerate}
		\end{defin}
		
		Let $X = (B_1 \times B_2)/G$ a BdF-variety of product type and write as usual $B_i = V_i/\Lam_i$. In this situation, $\Lam_2$ is a $G$-module, thus a module over the ring $\ZZ[G] \cong \ZZ[X]/(X^m-1)$. By the Chinese Remainder theorem, we find that $\ZZ[G]$ embeds  into a direct sum of cyclotomic rings, i.e., we have
		\begin{align*}
		\ZZ[X]/(X^m-1)  \subset \bigoplus_{k|m} \ZZ[X]/(\phi_k(X)),
		\end{align*}
		where $\phi_k(X)$ is the $k$-th cyclotomic polynomial. 

			\begin{rem}
				Fabrizio Catanese (private communication) pointed out that the inclusion above is not an isomorphism. In \cite{Fabrizio}, he inadvertently claimed  that exactly this was the case, while in fact, this is only true over $\QQ$. Throughout this paper we only deal with the case where $\Lam_2$ is as well a module over the direct sum $\bigoplus_{k|m} \ZZ[X]/(\phi_k(X))$. In the general case, one obtains tori which are isogenous to ones with splitting lattices as above (we shall deal with this question with more precision in a future paper).
			\end{rem}

		We write $R := \ZZ[x]/(x^m-1)$ and $R_k$ for the cyclotomic ring $\ZZ[x]/(\phi_k(x))$. We assume that $\Lam_2$ splits according to the order of the eigenvalues, $\Lam_2 = \bigoplus_{k|m} \Lam_{2,k}$ (note that  $\Lam_{2,k}$ is an $R_k$-module). The vector space $V_2$ splits accordingly, $V_2 = \bigoplus_{k|m} V_{2,k}$. We find that
		\begin{align} \label{splitting}
		B_2 = \bigoplus_{k|m} B_{2,k},
		\end{align}
		where $B_{2,k}$ is a $G$-invariant abelian subvariety of $B_2$ such that a generator $g \in G$ (as in Theorem \ref{charac}) acts on $B_{2,k}$ with eigenvalues of order $k$.

		\begin{defin}
		We call the abelian variety $B_2$ \textit{split} if it admits a splitting as in \ref{splitting}. A BdF variety $X = (B_1 \times B_2)/(G \times T)$ is called \textit{split} if $B_2$ is split.
		\end{defin}

		Let $X$ be a BdF-variety of dimension $n>1$. We furthermore assume that $X = (B_1 \times B_2)/G$ is a primary BdF-variety. In this case, $\Lam_2$ is a module over the ring $R_m = \ZZ[x]/(\phi_m(x))$, which is a Dedekind domain (see for example theorem 2.6 in \cite{Washington}). In fact, $\Lam_2$ is a projective $R_m$-module. We therefore get a splitting $\Lam_2 = R_m^r \oplus I$ of $\Lam_2$ into a free part and an ideal $I \subset R_m$ (see \cite{Milnor}). Indeed, $\Lam_2$ is free if $R_m$ is a PID, i.e., if the class number $h(R_m) := \# \Cl(R_m)$ is equal to $1$. We know that $R_m$ is a free $\ZZ$-module of rank $\fie(m)$ (where $\fie$ denotes the Euler totient function), hence we have
		\begin{align} \label{eqrank}
		\Lam_2 = R_m^r \cong \ZZ^{\fie(m)\cdot r}
		\end{align}
		for all $m$ satisfying $h(R_m) = 1$. As $B_1$ is an abelian variety of positive dimension, we have the following important observation:
		\begin{align} \label{ineq}
		\fie(m) \leq 2(n-1).
		\end{align}
		
		\begin{defin} Let $\torus = V/\Lam$ be a complex torus. If $\Lam = R_m^k$, then $\torus$ is called \textit{elementary}.
		\end{defin}
		
		We have the following well-known table of values $\fie(m)$: \\
		\begin{center}
		\begin{tabular}{lr}
		\begin{tabular}{|c|c|} \hline
		$\fie(m)$ & $m$\\ \hline \hline 
		$1$ & $1,2$ \\ \hline
		$2$ & $3,4,6$  \\ \hline
		$4$ & $5,8,10,12$ \\ \hline
		\end{tabular}
		&	
		\begin{tabular}{|c|c|} \hline
		$\fie(m)$ & $m$ \\ \hline \hline
		$6$ & $7,9,14,18$ \\ \hline
		$8$ & $15,16,20,24,30$ \\ \hline
		$10$ & $11,22$ \\ \hline
		\end{tabular}			
		\end{tabular}				
		\end{center}
		
		According to the table on page 353 of \cite{Washington}, we know in particular that $h(R_m) = 1$ for all values of $\fie(m)$ listed above.
		
		In the following sections, we will classify BdF-varieties in small dimensions. To obtain a satisfying classification, one first determines all possibilities for $B_2$. To give a complex structure to $(\Lam_2 \otimes_\ZZ \RR)/\Lam_2$ it is sufficient and necessary to give a decomposition $\Lam_2 \otimes_\ZZ \CC = V \oplus \overline{V}$ (see \cite{Griffiths-Harris}, pages 326 and 327). This amounts to determining all possible $G$-Hodge decompositions corresponding to $\Lam_2$. 
		
		\begin{rem} 
			\noindent \begin{enumerate}
				\item Note that, by the above inequality \ref{ineq}, the case $m=2$ obviously occurs for all dimensions $n>1$. The only automorphism of order $2$ of an abelian variety with finitely many fixed points is multiplication with $-1$, so the case $m=2$ can be easily classified. In what follows, this trivial case will be omitted during the calculations, but will be listed in the tables. (Note that the number of parameters for this case is given by the dimension of the Siegel upper half space.)
				\item In the following sections, we classify families of complex tori admitting a faithfully acting linear automorphism of a certain order up to dimension $5$ \textit{up to complex conjugation}. Usually, we will drop the phrase 'up to complex conjugation' in the statement of our results.
				\item We refrain from listing the well-known table of elliptic curves admitting automorphisms of finite order.
			\end{enumerate}
		\end{rem}

		\subsection{The Surface Case} \label{surface-case}
In this and the upcoming sections, we classify families of abelian varieties admitting a linear automorphism of certain order acting faithfully. To abbreviate, we denote by $\cond_{n,m}$ the following condition to be satisfied by a torus $T$:
\begin{center}
\begin{tabular}{lp{10cm}}
$\cond_{n,m}$ & $\dim(T) = n$, and  $T$ admits a linear automorphism of some order $m > 2$ having only primitive $m$-th roots of unity as eigenvalues
\end{tabular}
\end{center}

	For now, we assume that $\torus$ satisfies $\cond_{2,m}$. We write $\Lam \otimes_\ZZ \CC = V \oplus \overline{V}$. Using inequality \ref{ineq}, we find that we only have two possibilities in equation \ref{eqrank}: \par  \hfill \par
		\noindent \textit{Case 1:} $\fie(m) = 2$ (hence $m \in \{3,4,6\}$) and $r = 2$, or \\
		\textit{Case 2:} $\fie(m) = 4$ (hence $m \in \{5,8,10,12\}$) and $r = 1$. \par  \hfill \par
		
		In the first case, $R_m$ has rank $2$ as a $\ZZ$-module. We have the decomposition $\Lam \otimes_\ZZ \CC = W_\chi \oplus W_{\overline{\chi}}$ into $2$-dimensional isotypical components. Recall that we have a decomposition $\Lam \otimes_\ZZ \CC = V \oplus \overline{V}$. As we have seen in chapter 2, we also get a decomposition for $V$ (i.e., we have $\Lam \otimes_\ZZ \CC = V_\chi \oplus \overline{V_{\overline{\chi}}} \oplus V_{\overline{\chi}} \oplus \overline{V_\chi}$). We will always (i.e., in any dimension) distinguish between the cases where the following assumption holds (or does not hold).
		\begin{equation} \label{assump}
		\text{Only pairwise non-conjugate characters appear in } V.
		\end{equation}
		
The result \cite[Proposition 1.8, a)]{BGL} together with \ref{eqrank} now give the following
		
		\begin{prop} \label{assump-cor} If $\torus = V/\Lam$ is elementary (of arbitrary dimension) such that the rank of $\Lam$ as an $R_m$-module is $1$, then \ref{assump} holds. 
		\end{prop}

Note that the case where \ref{assump} holds is dealt with in \cite{Catanese-Ciliberto}; according to the authors, the classification of two-dimensional tori which admit an automorphism of order $m$ acting faithfully (such that \ref{assump} holds) is as follows. There are exactly
\begin{itemize}
\item $m=3,6$: one isomorphism class, namely $E_\rho \times E_\rho$ ($E_\rho$ being the equianharmonic elliptic curve).
\item $m=4$: one isomorphism class, namely $E_{\ci} \times E_{\ci}$ ($E_{\ci}$ being the harmonic elliptic curve).
\item $m=5, 10$: one isomorphism class $S_{10}$.
\item $m=8,12$: two isomorphism classes $S_m'$, $S_m''$.
\end{itemize}
Moreover Catanese and Ciliberto prove:

\begin{prop}
Let $\torus$ satisfy $\cond_{n,m}$ and $\ref{assump}$. Then the following hold.
\begin{enumerate}
\item If $m=3,6$, then $T \cong E_\rho^n$.
\item If $m=4$, then $T \cong E_{\ci}^n$.
\item If $m=5,10$, then $T \cong S_{10}^{n/2}$.
\item If $m=8,12$, then there are $k, l \in \NN$ such that $T \cong (S_m')^k \times (S_m'')^l$.
\end{enumerate}
\end{prop}
		
		
				
		
		
		
		
Now we analyze the case where \ref{assump} does not hold.
		\begin{prop} \label{conj-case-dim2}
			Assume that $V$ splits into the two one-dimensional spaces $V_\chi$ and $V_{\overline{\chi}}$, i.e., \ref{assump} does not hold. Then there is a two-parameter family of tori $T$ satisfying $\cond_{2,m}$. 
		\end{prop}
We denote these families by $S_4$ for $m=4$ and $S_6$ for $m \in \{3,6\}$.
		\begin{proof}
			It suffices to show that under the given conditions, one has a two-parameter family of complex tori modulo a group $G$ with the desired properties. This is clear, as we have only the possibility $V = V_\chi \oplus V_{\overline{\chi}}$ (up to complex conjugation) to decompose $V$. But $(V_\chi, V_{\overline{\chi}})$ is contained in the open set $\{(W,W') \in \Gr(1,2) \times \Gr(1,2)|W \cap W' = \{0\} \}$,
			which is two-dimensional.
		\end{proof}

We sum up our results in the following

		\begin{theorem} \label{surface_case}
			There are exactly $19$ families (up to complex conjugation) of two-dimensional complex tori admitting a non-trivial linear automorphism of finite order acting faithfully. These families are listed in \textsc{Table 2} below. By $p$, we denote the number of moduli of the corresponding family. The rest of the notation used is explained below the table.
			\begin{center}
				\small
				\textsc{Table 1} \par \hfill \par
				\resizebox{1.1\linewidth}{!}{
				\begin{tabular}{lcr}
				\begin{tabular}{|c|c|cH|c|} \hline \label{table2}
					$m$ & $\torus$ & $\Fix(g)$ & Types & $p$ \\ \hline \hline
					$2$ & $S$ & $\cgb4^8$ & $(2)$ & $3$ \\ \hline
					$3$ & $S_6$ & $\cgb3^2$ & $(3)$ & $2$ \\ \hline
					$3$ & $E_\rho \times E_\rho$ & $\cgb3^2$ & $(3,3)$ & $0$ \\ \hline 
					$4$ & $S_4$ & $\cgb2^2$ & $(4)$ & $2$ \\ \hline
					$4$ & $E_{\ci} \times E_{\ci}$ & $\cgb2^2$ & $(4,4)$ & $0$ \\ \hline
					$4$ & $E \times E_{\ci}$ & $\cgb2^3$ & $(2,4)$ & $1$\\ \hline 
					$5$ & $S_{10}$ & $\cg5$ & $(5)$ & $0$ \\ \hline
				\end{tabular}
				&
				\begin{tabular}{|c|c|cH|c|} \hline
					$m$ & $\torus$ & $\Fix(g)$ & Types & $p$ \\ \hline \hline
					$6$ & $S_6$ & $\{0\}$ & $(6)$ & $2$ \\ \hline
					$6$ & $E \times E_\rho$ & $\cgb2^2 \times \cg3$ & $(2,3)$ & $1$ \\ \hline 
			        $6$ & $E \times E_\rho$ & $\cgb2^2$ & $(2,6)$ & $1$ \\ \hline 
					$6$ & $E_\rho \times E_\rho$ & $\cg3$ & $(3,6)$ & $0$ \\ \hline 
					$6$ & $E_\rho \times E_\rho$ & $\{0\}$ & $(6,6)$ & $0$ \\ \hline 
					$8$ & $S_8'$ & $\cgb2^2$ & $(8)$ & $0$ \\ \hline
					$8$ & $S_8''$ & $\cgb2^2$ & $(8)$ & $0$ \\ \hline
				\end{tabular} &
				\begin{tabular}{|c|c|cH|c|} \hline
					$m$ & $\torus$ & $\Fix(g)$ & Types & $p$ \\ \hline \hline
					$10$ & $S_{10}$ & $\{0\}$ & $(10)$ & $0$ \\ \hline
					$12$ & $S_{12}'$ & $\{0\}$ & $(12)$ & $0$ \\ \hline
					$12$ & $S_{12}''$ & $\{0\}$ & $(12)$ & $0$ \\ \hline
					$12$ & $E_\rho \times E_{\ci}$ & $\cg2 \times \cg3$ & $(3,4)$ & $0$ \\ \hline 
					$12$ & $E_{\ci} \times E_\rho$ & $\cg2$ & $(4,6)$ & $0$ \\ \hline 
				\end{tabular}
				\end{tabular}}
				\normalsize
			\end{center}
		\end{theorem}

		The table is organized as follows. By $E$ resp. $S$ we denote an arbitrary elliptic curve resp. abelian surface, while $S_4$, $S_6$, $S_{10}$, $S_m'$, $S_m''$ denote the (families of) abelian surfaces introduced above. 
		We decided to list the same torus $\torus$ more than once if their entries in $\Fix(g)$ are different. Finally, $\Fix(g)$ denotes the fixed locus of a linear automorphism $g$ of $\torus$ of order $m$.
		
		\begin{rem}
		The result of Theorem \ref{surface_case} was already achieved by Fujiki \cite{Fujiki} with different methods, see also \cite{MTW} for a study in the case where $m$ is prime. 
		\end{rem}
		
		\subsection{The Threefold Case}
		
		Let $\torus$ satisfy $\cond_{3,m}$. The only possibilies are: \par  \hfill \par
		\noindent \textit{Case 1:} $\fie(m) = 2$ (hence $m \in \{3,4,6\}$) and $r = 3$, or \\
		\textit{Case 2:} $\fie(m) = 6$ (hence $m \in \{7,9,14,18\}$) and $r=1$. \par  \hfill \par
		
		We first deal with case 1. In the same way as in the surface case there is the decomposition $\Lam \otimes_\ZZ \CC = W_\chi \oplus W_{\overline{\chi}} = V_\chi \oplus \overline{V_{\overline{\chi}}} \oplus V_{\overline{\chi}} \oplus \overline{V_\chi}$. We have the following proposition:
		
		\begin{prop} \label{conj-case-dim3} Let $\torus$ satisfy $\cond_{3,m}$ and $\fie(m) = 2$. Then the following statements hold:
			\begin{enumerate}
				\item If \ref{assump} holds, then $\torus$ is isomorphic to a product of elliptic curves. More precisely, if $m \in \{3,6\}$, then $\torus = E_\rho \times E_\rho \times E_\rho$. If $m = 4$, then $\torus = E_{\ci} \times E_{\ci} \times E_{\ci}$.
				\item If \ref{assump} does not hold, there is a four parameter family of complex tori admitting a linear automorphism of order $m$ acting faithfully. 
			\end{enumerate}
		\end{prop}
We denote the respective families by $A_m$, $m \in \{4,6\}$. We also use the notation $A_6$ for the family obtained for $m=3$.
		\begin{proof}
			The first assertion is clear. For the second assertion, note that the only possible Hodge type is $(1,2)$, so the $G$-Hodge decompositions are parametrized by an open set in $\Gr(1,3) \times \Gr(2,3)$, which has dimension $4$.
		\end{proof}
		
		Case 2 is easy to analyze. Note that by Proposition \ref{assump-cor} we can assume that the eigenvalues of a generator $g \in G$ are distinct and pairwise non-complex conjugate.
		
		\begin{prop} \label{ab-threefolds} There are exactly two isomorphism classes of tori satisfying $\cond_{3,m}$ for any $m \in \{7,9,14,18\}$. We denote them by $A_7'$, $A_7''$ for $m \in \{7,14\}$ and by $A_9'$, $A_9''$ for $m \in \{9, 18\}$.
		\end{prop}
		
		\begin{proof}
			We only list one element of each orbit for $m=7,9$. For $m=7$, there are exactly two orbits corresponding to the tuples of pairwise non-complex conjugate characters $(1,2,3)$, $(1,2,4)$. For $m=9$, the two orbits correspond $(1,2,4)$ and $(1,4,7)$. 
		\end{proof}

		Now, our assumption \ref{splitting} guarantees that we can give a list of all isomorphism classes of abelian threefolds admitting an automorphism of finite order. Here, $\torus$ denotes an arbitrary abelian threefold. Again, we refrain from listing the same tori more than once unless their entries in the $\Fix(g)$-column are different. The rest of the notation is as in \textsc{Table 1}.
		\begin{theorem}
			There are exactly $56$ families (up to complex conjugation) of three-dimensional complex tori admitting a non-trivial linear automorphism of finite order acting faithfully. These families are listed in \textsc{Table 2} below. Again, $p$ denotes the number of moduli.
			\begin{center}
				\textsc{Table 2} \\
				\resizebox{1.1\linewidth}{!}{
					\begin{tabular}{lcr}
		
						\begin{tabular}{|c|c|cH|c|} \hline
							$m$ & $\torus$ & $\Fix(g)$ & Types & $p$ \\ \hline \hline
							$2$ & $\torus$ & $\cgb2^6$ & $(2)$ & $6$ \\ \hline
							$3$ & $A_6$ & $\cgb3^3$ & $(3)$ & $4$ \\ \hline
							$3$ & $E_\rho \times S_6$ & $\cgb3^3$ & $(3,3)$ & $2$ \\ \hline
							$3$ & $E_\rho \times E_\rho \times E_\rho$ & $\cgb3^3$ & $(3,3,3)$ & $0$ \\ \hline
							$4$ & $A_4$ & $\cgb2^3$ & $(4)$ & $4$ \\ \hline
							$4$ & $E \times S_4$ & $\cgb2^4$ & $(2,4)$ & $3$ \\ \hline
							$4$ & $E_{\ci} \times S$ & $\cgb2^3$ & $(4,2)$ & $3$ \\ \hline
							$4$ & $E_{\ci} \times S_4$ & $\cgb2^3$ & $(4,4)$ & $2$ \\ \hline
							$4$ & $E \times E_{\ci} \times E_{\ci}$ & $\cgb2^4$ & $(2,4,4)$ & $1$ \\ \hline
							$4$ & $E_{\ci} \times E_{\ci} \times E_{\ci}$ & $\cgb2^3$ & $(4,4,4)$ & $0$ \\ \hline
							$6$ & $A_6$ & $\{0\}$ & $(6)$ & $4$ \\ \hline
							$6$ & $E \times S_6$ & $\cgb2^2 \times \cgb3^2$ & $(2,3)$ & $3$ \\ \hline
							$6$ & $E \times S_6$ & $\cgb2^2$ & $(2,6)$ & $3$ \\ \hline
							$6$ & $E_\rho \times S_6$ & $\cgb3^2$ & $(6,3)$ & $2$ \\ \hline 
							$6$ & $E_\rho \times S_6$ & $\cgb3$ & $(3,6)$ & $2$ \\ \hline 
							$6$ & $E_\rho \times S_6$ & $\{0\}$ & $(6,6)$ & $2$ \\ \hline 
							$6$ & $E_\rho \times S$ & $\{0\}$ & $(6,2)$ & $3$ \\ \hline
							$6$ & $E_\rho \times S$ & $\cg3\times \cgb2^4$ & $(3,2)$ & $3$ \\ \hline
							$6$ & $E \times E_\rho \times E_\rho$ & $\cgb2^2 \times \cgb3^2$ & $(2,3,3)$ & $1$ \\ \hline
						\end{tabular}
						&
						\begin{tabular}{|c|c|cH|c|} \hline
							$m$ & $\torus$ & $\Fix(g)$ & Types & $p$ \\ \hline \hline
							$6$ & $E \times E_\rho \times E_\rho$ & $\cgb2^2 \times \cg3$ & $(2,3,6)$ & $1$ \\ \hline
							$6$ & $E \times E_\rho \times E_\rho$ & $\cgb2^2$ & $(2,6,6)$ & $1$ \\ \hline 
							$6$ & $E_\rho \times E_\rho \times E_\rho$ & $\cgb3^2$ & $(3,3,6)$ & $0$ \\ \hline
							$6$ & $E_\rho \times E_\rho \times E_\rho$ & $\cg3$ & $(3,6,6)$ & $0$ \\ \hline  
							$6$ & $E_\rho \times E_\rho \times E_\rho$ & $\{0\}$ & $(6,6,6)$ & $0$  \\ \hline  
							$7$ & $A_7'$ & $\cg7$ & $(7)$ & $0$ \\ \hline
							$7$ & $A_7''$ & $\cg7$ & $(7)$ & $0$ \\ \hline
							$8$ & $E \times S_8'$ & $\cgb2^3$ & $(2,8)$ & $1$ \\ \hline
							$8$ & $E \times S_8''$ & $\cgb2^3$ & $(2,8)$ & $1$ \\ \hline
							$8$ & $E_{\ci} \times S_8'$ & $\cgb2^2$ & $(4,8)$ & $0$ \\ \hline
							$8$ & $E_{\ci} \times S_8''$ & $\cgb2^2$ & $(4,8)$ & $0$ \\ \hline
							$9$ & $A_9'$ & $\cg3$ & $(9)$ & $0$ \\ \hline 
							$9$ & $A_9''$ & $\cg3$ & $(9)$ & $0$ \\ \hline
							$10$ & $E \times S_{10}$ & $\cg2 \times \cg5$ & $(2,5)$ & $1$ \\ \hline
							$12$ & $E_\rho \times S_4$ & $\cg3 \times \cgb2^2$ & $(3,4)$ & $2$ \\ \hline
							$12$ & $E_\rho \times S_4$ & $\cgb2^2$ & $(6,4)$ & $2$ \\ \hline
							$12$ & $E_{\ci} \times S_6$ & $\cg2 \times \cgb3^2$ & $(4,3)$ & $2$ \\ \hline
							$12$ & $E_{\ci} \times S_6$ & $\cg2$ & $(4,6)$ & $2$ \\ \hline
							$12$ & $E \times S_{12}'$ & $\{0\}$ & $(2,12)$ & $1$ \\ \hline
						\end{tabular}
						&
						\begin{tabular}{|c|c|cH|c|} \hline
							$m$ & $\torus$ & $\Fix(g)$ & Types & $p$ \\ \hline \hline
							$12$ & $E \times S_{12}''$ & $\{0\}$ & $(2,12)$ & $1$ \\ \hline
							$12$ & $E_\rho \times S_{12}'$ & $\cg3$ & $(3,12)$ & $0$ \\ \hline
							$12$ & $E_\rho \times S_{12}'$ & $\{0\}$ & $(6,12)$ & $0$ \\ \hline 
							$12$ & $E_\rho \times S_{12}''$ & $\cg3$ & $(3,12)$ & $0$ \\ \hline
							$12$ & $E_\rho \times S_{12}''$ & $\{0\}$ & $(6,12)$ & $0$ \\ \hline 
							$12$ & $E_{\ci} \times S_{12}'$ & $\cg2$ & $(4,12)$ & $0$ \\ \hline
							$12$ & $E_{\ci} \times S_{12}''$ & $\cg2$ & $(4,12)$ & $0$ \\ \hline
							$12$ & $E \times E_\rho \times E_{\ci}$ & $\cgb2^2 \times \cg3$ & $(2,3,4)$ & $1$ \\ \hline
							$12$ & $E \times E_{\ci} \times E_\rho$ & $\cgb2^3$ & $(2,4,6)$ & $1$ \\ \hline 
							$12$ & $E_\rho \times E_\rho \times E_{\ci}$ & $\cgb3^2 \times \cg2$ & $(3,3,4)$ & $0$ \\ \hline 
							$12$ & $E_\rho \times E_{\ci} \times E_{\ci}$ & $\cg3 \times \cgb2^2$ & $(3,4,4)$ & $0$ \\ \hline
							$12$ & $E_\rho \times E_{\ci} \times E_\rho$ & $\cgb3^2 \times \cg2$ & $(3,4,6)$ & $0$ \\ \hline
							$12$ & $E_{\ci} \times E_\rho \times E_\rho$ & $\cg2$ & $(4,6,6)$ & $0$ \\ \hline 
							$12$ & $E_{\ci} \times E_{\ci} \times E_\rho$ & $\cg2$ &$(4,4,6)$ & $0$ \\ \hline 
							$14$ & $A_7'$ & $\{0\}$ & $(14)$ & $0$ \\ \hline
							$14$ & $A_7''$ & $\{0\}$ & $(14)$ & $0$ \\ \hline
							$18$ & $A_9'$ & $\{0\}$ & $(18)$ & $0$ \\ \hline
							$18$ & $A_9''$ & $\{0\}$ & $(18)$ & $0$ \\ \hline
						\end{tabular}
					\end{tabular} }
				\end{center}
			\end{theorem}
			
			\subsection{The Fourfold Case} \label{4dim}
			Let $\torus$ satisfy $\cond_{4,m}$. As in the previous sections, we determine all possibilities for the order $m$ of $G$: \par  \hfill \par
			\noindent \textit{Case 1:} $\fie(m) = 2$ (hence $m \in \{3,4,6\}$) and $r = 4$, or \\
			\textit{Case 2:} $\fie(m) = 4$ (hence $m \in \{5,8,10,12\}$) and $r = 2$, or \\
			\textit{Case 3:} $\fie(m) = 8$ (hence $m \in \{15,16,20,24,30\}$) and $r = 1$. \par \hfill \par
			
			Most of the proofs of this section are very similar to the proofs given in the previous sections; therefore, we will omit most of them.
			We start with the first case.
			
			\begin{prop} Let $\torus$ satisfy $\cond_{4,m}$ such that $\fie(m) = 2$. Then the following statements hold:
				\begin{enumerate}
					\item If \ref{assump} holds, then $\torus$ is isomorphic to a product of elliptic curves. More precisely, if $m \in \{3,6\}$, then $\torus = E_\rho^4$. If $m = 4$, then $\torus = E_{\ci}^4$.
					\item If \ref{assump} does not hold, there are exactly two families of four-dimensional complex tori admitting an automorphism of order $m$, which have $6$ resp. $8$ parameters, respectively.
				\end{enumerate}
			\end{prop}
We denote the respective families in ii) by $X_m^6$, $X_m^8$, $m \in \{4,6\}$. We also use the notation $X_6^6$, $X_6^8$ for the families obtained for $m=3$.
			
			For case 2, we have
			
			\begin{prop} Let $\torus$ satisfy $\cond_{4,m}$ such that $\fie(m) = 4$. Then the following statements hold:
				\begin{enumerate}
					\item If \ref{assump} holds, then $\torus$ is isomorphic to one of the following complex tori: $(S_k')^2, S_k' \times S_k'', (S_k'')^2$.
					\item If \ref{assump} does not hold, there are exactly two families of complex tori of dimension $4$ admitting an automorphism of order $m$, which have $2$ resp. $4$ parameters.
				\end{enumerate}
			\end{prop}
We denote the respective families in ii) by $X_m^2$, $X_m^4$, $m \in \{8,10,12\}$. We also use the notation $X_{10}^2$, $X_{10}^4$ for the families obtained for $m=5$.

			Finally, the third case is analyzed. Again, we have that \ref{assump} holds by Proposition \ref{assump-cor}.
			
			\begin{prop} There are exactly four isomorphism classes of $4$-dimensional complex tori satisfying $\cond_{4,m}$ for each $m \in \{15, 16, 20, 30\}$ and exactly five isomorphism classes when $m = 24$.
			\end{prop}
			
			\begin{proof}
				We only list one element of each orbit for $m=15,16,20,24$: for $m=15$, the orbits correspond to the $4$-tuples of pairwise non-complex conjugate characters $(1,2,4,7)$, $(1,2,4,8)$, $(1,2,7,11)$, $(1,4,7,13)$. For $m=16$, the orbits correspond to $(1,3,5,9)$, $(1,3,5,7)$, $(1,3,9,11)$, $(1,5,9,13)$, while for $m=20$, they correspond to $(1,3,7,9)$, $(1,3,7,11)$, $(1,3,11,13)$, $(1,9,13,17)$. Finally, for $m=24$, the tuples $(1,5,7,11)$, $(1,5,7,13)$, $(1,5,13,17)$, $(1,7,13,19)$, $(1,11,17,19)$ give the five orbits.
			\end{proof}
			
			We denote representatives of the respective isomorphism classes for $m \in \{15,30\}$ by $X_{30}^{(i)}$, $i \in \{1,2,3,4\}$. For $m \in \{16,20\}$, the representatives are denoted $X_m^{(i)}$, $i \in \{1,2,3,4\}$. Finally, for $m = 24$, the representatives of the five isomorphism classes are denoted $X_{24}^{(j)}$, with $j \in \{1,2,3,4,5\}$. \par \hfill \par 
			
			We now could - in principle - give a list of all $4$-dimensional complex tori admitting a linear automorphism of finite order which have finite fixed locus. This will be omitted, as it is very tedious, yet not very enlightening. In lieu thereof, we will only list those $4$-dimensional complex tori which admit a linear automorphism of order $m$ with eigenvalues of order $m$. 
			
			\begin{theorem}
				The table below contains $4$-dimensional complex tori which admit a linear automorphism $g$ of order $m$, whose eigenvalues are only primitive $m$-th roots of unity. Conversely, every such complex torus belongs (up to complex conjugation) to one of the families listed below. \\
							Here, $\tilde\torus$ denotes an arbitrary $4$-dimensional complex torus. The remaining notation is as introduced above.
				\begin{center}
					\textsc{Table 3} \\
					
					\renewcommand{\arraystretch}{1.25}
					\resizebox{1.1\linewidth}{!}{
						\begin{tabular}{lcr}
							\begin{tabular}{|c|c|c|c|} \hline
								$m$ & $\torus$ & $\Fix(g)$ & $p$ \\ \hline \hline
								$2$ & $\tilde \torus$ & $\cgb2^8$ & $10$ \\ \hline
								$3$ & $E_\rho^4 $& $\cgb3^4$ & $0$ \\ \hline
								$3$ & $E_\rho \times E_\rho \times S_6$ & $\cgb3^4$ & $2$ \\ \hline
								$3$ & $S_6 \times S_6$ & $\cgb3^4$ & $4$ \\ \hline
								$3$ & $E_\rho \times A_6$ & $\cgb3^4$ & $4$ \\ \hline
								$3$ & $X_6^6$ & $\cgb3^4$ & $14$ \\ \hline
								$3$ & $X_6^8$ & $\cgb3^4$ & $14$ \\ \hline
								$4$ & $E_{\ci}^4$ & $\cgb2^4$ & $0$ \\ \hline
								$4$ & $E_{\ci} \times E_{\ci} \times S_4$ & $\cgb2^4$ & $2$ \\ \hline
								$4$ & $S_4 \times S_4$ & $\cgb2^4$ & $4$ \\ \hline
								$4$ & $E_{\ci} \times A_4$ & $\cgb2^4$ & $4$ \\ \hline
								$4$ & $X_4^6$ & $\cgb2^4$ & $14$ \\ \hline
								$4$ & $X_4^8$ & $\cgb2^4$ & $14$ \\ \hline
								$5$ & $S_{10} \times S_{10}$ & $\cgb5^2$ & $0$\\ \hline
							\end{tabular}
							&
							\begin{tabular}{|c|c|c|c|} \hline
								$m$ & $\torus$ & $\Fix(g)$ & $p$ \\ \hline \hline

								$5$ & $X_{10}^2$ & $\cgb5^2$ & $6$\\ \hline
								$5$ & $X_{10}^4$ & $\cgb5^2$ & $6$\\ \hline
								$6$ & $E_\rho^4$ & $\{0\}$ & $0$ \\ \hline
								$6$ & $E_\rho \times E_\rho \times S_6$ & $\{0\}$ & $2$ \\ \hline
								$6$ & $S_6 \times S_6$ & $\{0\}$ & $4$ \\ \hline
								$6$ & $E_\rho \times A_6$ & $\{0\}$ & $4$ \\ \hline
								$6$ & $X_6^6$ & $\{0\}$ & $14$ \\ \hline
								$6$ & $X_6^8$ & $\{0\}$ & $14$ \\ \hline
								$8$ & $S_8' \times S_8'$ & $\cgb2^2$ & $0$ \\ \hline
								$8$ & $S_8' \times S_8''$ & $\cgb2^2$ & $0$ \\ \hline
								$8$ & $S_8'' \times S_8''$ & $\cgb2^2$ & $0$\\ \hline
								$8$ & $X_8^2$ & $\cgb2^2$ & $6$ \\ \hline
								$8$ & $X_8^4$ & $\cgb2^2$ & $6$ \\ \hline
								$10$ & $S_{10} \times S_{10}$ & $\{0\}$ & $0$ \\ \hline
							\end{tabular} 
							&
							\begin{tabular}{|c|c|c|c|} \hline
								$m$ & $\torus$ & $\Fix(g)$ & $p$ \\ \hline \hline
								$10$ & $X_{10}^2$ & $\{0\}$ & $6$ \\ \hline
								$10$ & $X_{10}^4$ & $\{0\}$ & $6$ \\ \hline
								$12$ & $S_{12}' \times S_{12}'$ & $\{0\}$ & $0$ \\ \hline
								$12$ & $S_{12}' \times S_{12}''$ & $\{0\}$ & $0$\\ \hline
								$12$ & $S_{12}'' \times S_{12}''$ & $\{0\}$ & $0$\\ \hline
								$12$ & $X_{12}^2$ & $\{0\}$ & $6$ \\ \hline
								$12$ & $X_{12}^4$ & $\{0\}$ & $6$ \\ \hline
								$15$ & $X_{30}^{(i)}$, $i=1,2,3,4$ & $\{0\}$ & $0$ \\ \hline
								$16$ & $X_{16}^{(i)}$, $i=1,2,3,4$ & $\cg2$ & $0$ \\ \hline
								$20$ & $X_{20}^{(i)}$, $i=1,2,3,4$ & $\{0\}$ & $0$ \\ \hline
								$24$ & $X_{24}^{(i)}$, $i=1,...,5$ & $\{0\}$ & $0$ \\ \hline
								$30$ & $X_{30}^{(i)}$, $i=1,2,3,4$ & $\{0\}$ & $0$ \\ \hline
							\end{tabular}
						\end{tabular}}
					\end{center}
				\end{theorem}
				
				\subsection{The Fivefold Case} \label{5dim}
				Let $T$ satisfy $\cond_{5,m}$.
				We have the following possibilities:
				
				\textit{Case 1:} $\fie(m) = 2$ (hence $m \in \{3,4,6\}$) and $r = 5$, or \\
				\textit{Case 2:} $\fie(m) = 10$ (hence $m \in \{11,22\}$) and $r = 1$. \\
				
				\begin{prop} Let $T$ satisfy $\cond_{5,m}$ such that $\fie(m) = 2$. Then:
					\begin{enumerate}
						\item If \ref{assump} holds, then $\torus$ is isomorphic to a product of elliptic curves.
						\item If \ref{assump} does not hold, then there are exactly two families of such abelian varieties, which have $8$ resp. $12$ parameters. 
					\end{enumerate}
				\end{prop}
				
				The analysis of the second case yields
				
				\begin{prop} There are exactly four isomorphism classes of tori satisfying $\cond_{5,m}$ for $m \in \{11, 22\}$. 
				\end{prop}
				
				\begin{proof}
					It suffices to deal with the case $m=11$. The four orbits correspond to the $5$-tuples of pairwise non-complex conjugate characters $(1,2,3,4,5)$, $(1,2,3,4,6)$, $(1,2,3,5,7)$, $(1,3,4,5,9)$.
				\end{proof}
				
				Writing a table containing all families of $5$-dimensional complex tori which admit a linear automorphism of order $m$ acting faithfully is omitted, as the table would need to much space. 
				
				\renewcommand{\arraystretch}{1}
				\section{Classification of BdF-Varieties in Small Dimensions}
				In this chapter, we list all candidates for split BdF-varieties of dimensions $\leq 4$. It suffices to give a list of $m$, $B_1$, $B_2$ and $T_{r,i}$ with our desired properties (cf. Theorem \ref{charac}). As we saw, $B_2$ has to admit an automorphism of order $m$: all possibilities for $B_2$ up to dimension $3$ with corresponding $m$ were listed in the last section. An arbitrary abelian variety of a suitable dimension can be chosen for $B_1$. To compute all possibilities for $T$, we use the combinatorical restrictions, namely \cite[Corollary 1.7, Proposition 1.8]{BGL} (note that part (c) of 1.8 in \textit{loc. cit.} equals Proposition \ref{p-pow} in this paper), Remarks \ref{torsionpoints} and \ref{el-of-order-m}. We abbreviate $\ZZ_m := \ZZ/m\ZZ$. The notation is explained before the tables.
				
				\begin{rem} At this point, it is not clear that all the BdF-manifolds listed in the tables below are really projective, at least in dimension $> 2$. We will deal with this question in the next chapter.
				\end{rem} \par \hfill
				
				We start out with the surface case. We list all BdF-surfaces of product type together will all possibilities of $T_{r,i}$. 
				
				
				
				\begin{theorem}
					There are exactly seven families of BdF-surfaces. They are listed in \textsc{Table 4} below.
					\begin{center}
						\textsc{Table 4} \par \hfill \par
						\small
						\resizebox{\linewidth}{!}{
						\begin{tabular}{lr} 
						\begin{tabular}{|c|c|c|c|c|} \hline
							$m$ & $B_1$ & $B_2$ & Possibilities for $T_{r,i}$ & $p$ \\ \hline \hline
							$2$ & $E$ & $E'$ & $\{0\}, \ZZ_2$ & $2$ \\ \hline
							$3$ & $E$ & $E_\rho$ & $\{0\}, \ZZ_3$ & $1$ \\ \hline
						\end{tabular}
						&
						\begin{tabular}{|c|c|c|c|c|} \hline
						$m$ & $B_1$ & $B_2$ & Possibilities for $T_{r,i}$ & $p$ \\ \hline \hline
						$4$ & $E$ & $E_{\ci}$ & $\{0\}, \ZZ_2$ & $1$\\ \hline
						$6$ & $E$ & $E_\rho$ & $\{0\}$ & $1$ \\ \hline
						\end{tabular}
						\end{tabular}}
						\normalsize
					\end{center}
					Here, $p$ denotes the number of moduli.
				\end{theorem}
				
				This is the classification result of Bagnera-de Franchis \cite{BdF} and Enriques-Severi \cite{Enr-Sev}.
				
				
				\begin{theorem}
					Every threedimensional BdF-variety belongs to a family which is listed in \textsc{Table 5} below.
					\begin{center}
						\textsc{Table 5} \\
						\small
						\resizebox{1.1\linewidth}{!}{
						\begin{tabular}{lr}
						\begin{tabular}{|cH|c|c|} \hline
							$m$ & $B_1$ & $B_2$ & Possibilities for $T_{r,i}$ \\ \hline \hline
							$2$ & $S$ & $E$ & $\{0\}, \ZZ_2, \ZZ_2^2$ \\ \hline
							$2$ & $E$ & $S$ & $\{0\}, \ZZ_2$ \\ \hline
							$3$ & $S$ & $E_\rho$ & $\{0\}, \ZZ_3$ \\ \hline
							$3$ & $E$ & $S_6$ & $\{0\}, \ZZ_3$ \\ \hline
							$3$ & $E$ & $E_\rho \times E_\rho$ & $\{0\}, \ZZ_3$ \\ \hline
							$4$ & $S$ & $E_{\ci}$ & $\{0\}, \ZZ_2$ \\ \hline
							$4$ & $E$ & $S_4$ & $\{0\}, \ZZ_2, \ZZ_2^2$ \\ \hline
							$4$ & $E$ & $E_{\ci} \times E_{\ci}$ & $\{0\}, \ZZ_2, \ZZ_2^2$ \\ \hline
							$4$ & $E$ & $E' \times E_{\ci}$ & $\{0\}, \ZZ_2, \ZZ_2^2$ \\ \hline
							$5$ & $E$ & $S_{10}$ & $\{0\}, \ZZ_5$ \\ \hline
						\end{tabular} &
						\begin{tabular}{|cH|c|c|} \hline
							$m$ & $B_1$ & $B_2$ & Possibilities for $T_{r,i}$ \\ \hline \hline
							$6$ & $S$ & $E_\rho$ & $\{0\}$ \\ \hline
							$6$ & $E$ & $S_6$ & $\{0\}$ \\ \hline
							$6$ & $E$ & $E' \times E_\rho$ & $\{0\}, \ZZ_2, \ZZ_3, \ZZ_2^2, \ZZ_6, \ZZ_2^2 \times \ZZ_3$ \\ \hline
			     			$6$ & $E$ & $E_\rho \times E_\rho$ & $\{0\}, \ZZ_3$ \\ \hline
							$8$ & $E$ & $S_8'$ & $\{0\}, \ZZ_2$ \\ \hline
							$8$ & $E$ & $S_8''$ & $\{0\}, \ZZ_2$ \\ \hline
							$10$ & $E$ & $S_{10}$ & $\{0\}$ \\ \hline
							$12$ & $E$ & $S_{12}'$ & $\{0\}$ \\ \hline
							$12$ & $E$ & $S_{12}''$ & $\{0\}$ \\ \hline
							$12$ & $E$ & $E_{\ci} \times E_\rho$ & $\{0\}, \ZZ_2, \ZZ_3, \ZZ_6$ \\ \hline
						\end{tabular}
							\end{tabular}}
					\end{center}
				\end{theorem}
				\normalsize
				
				\begin{theorem}
					Every fourdimensional BdF-variety belongs to a family which is listed in \textsc{Table 6} below.
					\begin{center}
						\textsc{Table 6} \\
						\small
						\resizebox{1.1\linewidth}{!}{
						\begin{tabular}{llrr}
							\begin{tabular}{|cH|c|} \hline
								$m$ & $B_1$ & $B_2$ \\ \hline \hline
								$2$ & $X$ & $E$ \\ \hline
								$2$ & $S$ & $S'$ \\ \hline
								$2$ & $E$ & $X$ \\ \hline
								$3$ & $X$ & $E_\rho$ \\ \hline
								$3$ & $S$ & $S_6$ \\ \hline
								$3$ & $S$ & $E_\rho \times E_\rho$ \\ \hline
								$3$ & $E$ & $A_6$ \\ \hline
								$3$ & $E$ & $E_\rho \times S_6$\\ \hline
								$3$ & $E$ & $E_\rho \times E_\rho \times E_\rho$ \\ \hline
								$4$ & $X$ & $E_{\ci}$\\ \hline
								$4$ & $S$ & $S_4$ \\ \hline
								$4$ & $S$ & $E_{\ci} \times E_{\ci}$\\ \hline
								$4$ & $S$ & $E \times E_{\ci}$ \\ \hline
								$4$ & $E$ & $A_4$ \\ \hline 
								$4$ & $E$ & $E_{\ci} \times S$ \\ \hline 
								
							\end{tabular} &
							\begin{tabular}{|cH|c|} \hline
								$m$ & $B_1$ & $B_2$ \\ \hline \hline
								$4$ & $E$ & $E_{\ci} \times S_4$ \\ \hline 
								$4$ & $E$ & $E' \times S_4$ \\ \hline 
								$4$ & $E$ & $E' \times E_{\ci} \times E_{\ci}$ \\ \hline 
								$4$ & $E$ & $E_{\ci} \times E_{\ci} \times E_{\ci}$ \\ \hline
								$5$ & $S$ & $S_{10}$\\ \hline
								$6$ & $X$ & $E_\rho$ \\ \hline
								$6$ & $S$ & $S_6$\\ \hline
								$6$ & $S$ & $E \times E_\rho$ \\ \hline
								$6$ & $S$ & $E_\rho \times E_\rho$ \\ \hline
								$6$ & $E$ & $A_6$ \\ \hline
								$6$ & $E$ & $E' \times S_6$ \\ \hline
								$6$ & $E$ & $E_\rho \times S_6$ \\ \hline
								$6$ & $E$ & $E_\rho \times S$ \\ \hline
								$6$ & $E$ & $E' \times E_\rho \times E_\rho$ \\ \hline
								$6$ & $E$ & $E_\rho \times E_\rho \times E_\rho$ \\ \hline
							\end{tabular}
							&
							\begin{tabular}{|cH|c|} \hline
								$m$ & $B_1$ & $B_2$ \\ \hline \hline
								$7$ & $E$ & $A_7'$\\ \hline
								$7$ & $E$ & $A_7''$ \\ \hline
								$8$ & $S$ & $S_8'$ \\ \hline
								$8$ & $S$ & $S_8''$ \\ \hline
								$8$ & $E$ & $E' \times S_8'$ \\ \hline
								$8$ & $E$ & $E' \times S_8''$ \\ \hline
								$8$ & $E$ & $E_{\ci} \times S_8'$ \\ \hline
								$8$ & $E$ & $E_{\ci} \times S_8''$ \\ \hline
								$9$ & $E$ & $A_9'$ \\ \hline
								$9$ & $E$ & $A_9''$ \\ \hline
								$10$ & $S$ & $S_{10}$ \\ \hline
								$10$ & $E$ & $E' \times S_{10}$ \\ \hline
								$12$ & $S$ & $S_{12}'$ \\ \hline
								$12$ & $S$ & $S_{12}''$ \\ \hline
								$12$ & $S$ & $E_{\ci} \times E_\rho$ \\ \hline
							\end{tabular}
								&
								\begin{tabular}{|cH|c|} \hline
									$m$ & $B_1$ & $B_2$ \\ \hline \hline
								$12$ & $E$ & $E_\rho \times S_4$  \\ \hline
								$12$ & $E$ & $E_{\ci} \times S_6$ \\ \hline 
								$12$ & $E$ & $E' \times S_{12}'$ \\ \hline
								$12$ & $E$ & $E' \times S_{12}''$ \\ \hline
								$12$ & $E$ & $E_\rho \times S_{12}'$ \\ \hline
								$12$ & $E$ & $E_\rho \times S_{12}''$ \\ \hline
								$12$ & $E$ & $E_{\ci} \times S_{12}'$ \\ \hline
								$12$ & $E$ & $E_{\ci} \times S_{12}''$ \\ \hline
								$12$ & $E$ & $E' \times E_\rho \times E_{\ci}$ \\ \hline
								$12$ & $E$ & $E_\rho \times E_{\ci} \times E_{\ci}$ \\ \hline
								$12$ & $E$ & $E_\rho  \times E_\rho \times E_{\ci}$ \\ \hline
								$14$ & $E$ & $A_7'$ \\ \hline
								$14$ & $E$ & $A_7''$ \\ \hline
								$18$ & $E$ & $A_9'$ \\ \hline
								$18$ & $E$ & $A_9''$ \\ \hline
							\end{tabular}
						\end{tabular}}
					\end{center}
				\end{theorem}
				We leave it to the interested reader to write down all the possibilities for $T_{r,i}$ in \textsc{Table 6}.
				\normalsize
				\begin{rem} One could, in theory, write tables with all candidates for BdF-varieties up to dimension $11$ using the presented method; in fact, the first case where $\Lam_2$ is not a free $R_m$-module is $m=23$, meaning $\dim(B_2) = 11$.
				\end{rem}

				\section{Projectivity of BdF-manifolds}
				
				In the last section, we listed all candidates for BdF-varieties. However, the answer to the following question still remains.
				
				\begin{question} When do the families of BdF-manifolds listed in the tables above contain BdF-varieties?
				\end{question}
				
				In the sequel, we will prove that each family of BdF-manifolds contains a projective member using different methods. The first method is a theorem of T. Ekedahl, which briefly can be stated as 'rigid group actions on tori are projective', and whose proof was sketched to F. Catanese at an Oberwolfach conference. In the subsequent section, we find explicit forms for the polarization. These methods overlap somehow - nevertheless we think that it is of interest to present both methods.
				
				\subsection{Rigid Group Actions on Tori}
				
				 We state a result by Ekedahl; a detailed proof of it can be found in the preprint \cite{Demleitner} by F. Catanese and the author.
				
				\begin{theorem}[Ekedahl] \label{ekedahl}
				Let $(\torus,G)$ be a rigid group action of a finite group $G$ on a complex torus $\torus$. Then $\torus$ (or, equivalently, $\torus/G$) is projective.
				\end{theorem}
				
				Writing $T = V/\Lam$, the rigidity of a pair $(\torus,G)$ amounts to requiring that each character $\chi$ of $G$ appears in at most one of $V^{1,0}$ and $V^{0,1}$, so we can easily determine which complex tori in the above tables are projective:
				
				\begin{cor}
					The complex tori admitting a linear automorphism of order $m > 2$ acting faithfully for which assumption \ref{assump} holds are projective. In particular, they give rise to BdF-varieties.
				\end{cor}
				
				\subsection{Explicit Forms for the Polarization}
				
				In the following we find directly an explicit elementary form of the polarization, showing that each family of BdF-manifolds contains a projective member. Moreover, we determine the type of the polarization whenever possible.
				
				\begin{lemma} \label{m-odd} The following statements hold:
					There is no non-degenerate $G$-invariant alternating bilinear form on $R = \ZZ[X]/(X^m-1)$ for each $m > 2$. But there is such a form $\AltForm$ on $\ZZ[X]$ satisfying
					\begin{align*}
					\ker(\AltForm) = \left\{
					\begin{array}{l l}
					(X^{m-1} + X^{m-3} ... + 1), & m \text{ odd}  \\ 
					(X^{m-2} + X^{m-4} + ... + 1), & m \text{ even}.
					\end{array}
					\right.
					\end{align*}
					
				\end{lemma}
				
				\begin{proof}
					We only prove the lemma in the case where $m$ is odd, the other case is similar. Let $\AltForm$ be a form on $\AltForm$ as in the statement of the lemma. Then we have, by $G$-invariance, 
					\begin{center}
						$\AltForm(X^i,X^j) = \begin{cases} 0, & i = j, \\
						\AltForm(1,X^{j-i}), & j > i,
						\end{cases}$
					\end{center}
					and defining the rest by $\ZZ$-bilinearity and alternation. Setting $\la_i = \AltForm(1,X^i)$ for $1 < i < m$, we find
					\begin{equation}
					0 = \AltForm(X^i, X^m-1) = \AltForm(X^i,X^m) - \AltForm(X^i,1) = \la_{m-i} + \la_i.
					\end{equation}
					Hence, $\AltForm$ is uniquely determined by $\la_1, ..., \la_{\lfloor m/2 \rfloor}$, which are not all equal to $0$.
					By slight abuse of notation, we use the letter $\AltForm$ not only for the alternating form, but also for its matrix $\AltForm = \left(E(X^i,X^j)\right)_{ij}$.
					One sees that $X^{m-1}+...+1$ is always in the kernel of $\AltForm$, independent of the choice of the $\la_i$. \\
					Setting $\la_1 := 1$, $\la_{m-1} := -1$ and $\la_i := 0$ for $i \in \{2, ..., m-2\}$, one obtains a desired form as stated in the lemma.
				\end{proof}

				\begin{cor} \label{cor-n-even}
					The alternating form $\AltForm$ of the previous lemma induces a non-degenerate $G$-invariant bilinear form on the $\ZZ$-module $R_k = \ZZ[X]/(\phi_k(X))$ for a divisor $k$ of $n+1 = m$ for each $n > 1$.
					More precisely, $\AltForm$ can be uniquely written as a sum of non-degenerate $G$-invariant bilinear forms $\AltForm_k$ on $R_k$ for $k$ dividing $n+1$.
				\end{cor}
				
				\begin{proof}
					Write $R' = R_k \oplus R_k' = \ZZ[X]/(\phi_k(X)) \oplus \ZZ[X]/(\psi_k(X))$. It suffices to show that $\AltForm(R_k, R_k') = 0$. Note that $R_k$ and $R_k'$ decompose in direct sums of character-eigenspaces; first, let $r \in R_k, r' \in R_k'$ such that $r,r'$ belong to different character-eigenspaces. Hence the condition that $\AltForm$ is $G$-invariant yields
					\begin{align*}
					g\cdot \AltForm(r,r') = \chi(g) \overline{\psi}(g) \AltForm(r,r') = \AltForm(r,r'),
					\end{align*}
					with different characters $\chi \neq \psi$. Hence $\AltForm(r,r') = 0$. If $r \in R_k$, $r' \in R_k'$ are general elements (i.e., a sum of elements belonging to the respective character eigenspaces), we get $\AltForm(r,r') = 0$ by bilinearity.
				\end{proof}
				
				\begin{rem} \label{cplstructure}
					\begin{itemize}
					\item[i)]Recall that, if $\Lam = \ZZ[X]/(X^n+...+1)$ (for even $n$ and $m = n+1$), we have decompositions 
					\[
					\Lam \otimes \CC = V \oplus \overline{V} = \bigoplus_{\id \neq \chi \in G^\vee} W_{\chi} = \bigoplus_{j < \frac{m}{2}} W_{\chi_j} \oplus W_{\overline{\chi_j}}
					\]
					and $V = \bigoplus_\chi V_\chi$, such that, for each $\chi$, either $W_\chi = V_\chi$ or $W_{\overline{\chi}} = V_{\overline{\chi}}$. We fix the indices such that $\chi_j$ corresponds to the eigenvalue $\eps^j$ (where $\eps = \exp\left(\frac{2\pi i}{m}\right)$). One can do the same if $\Lam = \ZZ[X]/(X^{n-1} + X^{n-3}+...+1)$ in the case where $n$ is odd.
					Unless otherwise stated, we choose the complex structure such that $W_{\chi_j} = V_{\chi_j}$ for every $j < \frac{m}{2}$. 
					\item[ii)] From the above discussion it follows that a general alternating form $\AltForm$ as in Lemma \ref{m-odd} takes the following shape	
					\begin{align} \label{genform_A}
					\AltForm_{ij} = \begin{cases} 0, & i = j, \\
					\la_{j-i}, & j > i, \\
					-\la_{i-j}, & j < i.
					\end{cases}
					\end{align}
					\end{itemize}
				\end{rem}
				
				\begin{lemma} \label{case-odd-form} For each $n > 1$, there is a non-degenerate $G$-invariant alternating bilinear form $\AltForm$ on $R' = \ZZ[X]/(X^n+...+1)$ resp. $R' =  \ZZ[X]/(X^n+X^{n-2}...+1)$ (depending on the parity of $n$), such that the associated form $H \colon \CC^n \times \CC^n \to \RR$ defined by $H(v,w) = \AltForm(\ci \cdot v,w) + \ci \AltForm(v,w)$ and $\RR$-bilinear extension is Hermitian and positive definite.
				\end{lemma}
				
				\begin{proof}
					Let $m = n+1$. We only prove the lemma in the case where $n$ is even. Write $R = \ZZ \oplus R'$, where $R$ is as in Lemma \ref{m-odd} and let $\AltForm$ be given as in \ref{genform_A}. 
					It remains to show that the Riemann Bilinear Relations are satisfied. The first Riemann Bilinear Relation is satisfied if and only if $\AltForm(V,V) = 0$. Due to the previous corollary, it suffices to check this condition on eigenvectors: it is well-known that $V_{\chi_j}$ (for $j < \frac{m}{2}$) is generated by the element $v_j = \sum_{i = 0}^{n} \eps^{-ji} X^i \in R$. Then one computes:
					
					\begin{align*}
					\AltForm(v_h,v_k) &= \sum_{j=1}^n \sum_{i=0}^{n} \la_j \left((X^i)^\vee \otimes (X^{i+j})^\vee - (X^{i+j})^\vee \otimes (X^i)^\vee \right)(v_h,v_k)\\
					& = \sum_{j=1}^n \la_j (\eps^{-kj} - \eps^{-hj}) \sum_{i=0}^{n} \eps^{-hi-ki}.
					\end{align*}
					
					The last sum is $0$ unless $k+h \equiv 0 \pmod m$. But if $k + h \equiv 0 \pmod m$, we find that $v_k$ and $v_{-k} = \overline{v_k}$ do not both belong to $V$. Restricting the form to $R'$, this shows that the first Riemann Bilinear Relation is satisfied, independently of the choice of the $\la_i$. \\ 
					Now consider the special form $\AltForm$ obtained by setting $\la_1 = 1$, $\la_j = 0$ for every other $j$, and restricting to $R'$. By abuse of notation, we denote the images of the $v_k$ in $R'$ again by $v_k$. We prove the second Riemann Bilinear Relation for this form: we have to check that $\AltForm(- \ci v_k, \overline{v_k}) > 0$ for every $k < \frac{m}{2}$. A simple calculation gives
					\begin{align*}
					&\AltForm(- \ci v_k,\overline{v_k}) =  2n \cdot \sin\left(\frac{2 \pi k}{m}\right) > 0.
					\end{align*}
				\end{proof}
				
				Summarizing our results, we have the following:
				
				\begin{theorem}
					Let $\torus = V/\Lam$ be a complex torus and $G$ a cyclic group of order $m \geq 2$ of order $m$ acting freely on $\torus$ such that $\Lam$ is a cyclotomic submodule of $\ZZ[X]/(X^{m-1}+...+1)$ or $\ZZ[X]/(X^{m-2}+X^{m-4}+...+1)$ (depending on the parity of $m$). Assume furthermore that $V$ splits as in Remark \ref{cplstructure} i).  Then there is a non-degenerate $G$-invariant alternating bilinear form $\AltForm$ on $\Lam$ whose associated Hermitian form is positive definite, such that the polarization on $\torus$ is principal.                                                                                                                                                                                                                                                                                                                                                                                                                                                                                                                                                                                 
				\end{theorem}
				
				\begin{proof}
					The matrix of $\AltForm$ as chosen in the proofs above has determinant $1$. 
				\end{proof}
				
				\begin{rem}
					Indeed, we have shown in this sub-chapter that all of the following families of BdF-varieties $A= (B_1 \times B_2)/(T \times G)$ contain a projective member:
					\begin{itemize}
						\item The families where \ref{assump} is not satisfied (i.e., the non-rigid cases): In this case, a product of lower-dimensional abelian varieties for which \ref{assump} holds is always in the same family as $B_2$.
						\item The families where the complex structure is chosen as in Remark \ref{cplstructure} i) (or the one conjugate to it).
					\end{itemize}
				\end{rem}
				
				We briefly treat the exceptional cases which we not dealt with in the previous discussion, i.e., the ones where the complex structure is different from the one in Remark \ref{cplstructure}. Note that all these cases are rigid, hence projective by Ekedahl's Theorem \ref{ekedahl}; nevertheless, it is also of interest to give an explicit form of a polarization for these cases. One checks computationally that these values of $\la_i$ give rise to a positive definite form on $\ZZ[X]/(X^{m-1}+...+1)$ resp. $\ZZ[X]/(X^{m-2}+X^{m-4}+...+1)$.
				\begin{center}
					\textsc{Table 7} \par \hfill \par 
					\small
					\resizebox{1.1\linewidth}{!}{
					\begin{tabular}{lr}
					\begin{tabular}{|c|c|c|} \hline
						Case & Orbit & $\left(\la_1, ..., \la_{\lfloor m/2 \rfloor}\right)$ \\ \hline \hline 
						
						$S_8''$ & $(1,5)$ & $\left(-1,1,0\right)$ \\ \hline
						$S_{12}''$ & $(1,7)$ & $\left(-1,0,1^2,-1\right)$ \\ \hline \hline
						
						$A_7''$ & $(1,2,4)$ & $\left(0,-1,1\right)$ \\ \hline 
						$A_9''$ & $(1,2,7)$ & $\left(0, (-1)^2, 0\right)$ \\ \hline \hline
						
						$X_{15}^{2}$ & $(1,2,4,8)$ & $\left((-1)^4,1,0,1\right)$ \\ \hline
						$X_{15}^{3}$ & $(1,2,7,11)$ & $\left((-1)^3,0^2,(-1)^2\right)$ \\ \hline
						$X_{15}^{4}$ & $(1,4,7,13)$ & $\left(-1,1^2,(-1)^2,1,-1\right)$ \\ \hline \hline
						
						$X_{16}^{2}$ & $(1,3,5,9)$ & $\left(0^2,1,-1,0,-1,0^2\right)$ \\ \hline
						$X_{16}^{3}$ & $(1,3,9,11)$ & $\left(0,(-1)^7,0\right)$ \\ \hline
						$X_{16}^{4}$ & $(1,5,9,13)$ & $\left(0,-1,1,(-1)^2,1,-1,0^2\right)$ \\ \hline
					\end{tabular}
						&
					\begin{tabular}{|c|c|c|} \hline
				     	Case & Orbit & $\left(\la_1, ..., \la_{\lfloor m/2 \rfloor}\right)$ \\ \hline \hline 
						$X_{20}^{2}$ & $(1,3,7,11)$ & $\left((-1)^2,0,-1,0,-1,0,-1,0^2\right)$ \\ \hline
						$X_{20}^{3}$ & $(1,3,11,13)$ & $\left((-1)^2,0^8\right)$ \\ \hline
						$X_{20}^{4}$ & $(1,9,13,17)$ & $\left(1^2,0,(-1)^2,0^2,(-1)^2,0\right)$ \\ \hline \hline
						
						$X_{24}^{2}$ & $(1,5,7,13)$ & $\left(1^3,0,(-1)^3,0^2,1,0,-1,0\right)$ \\ \hline
						$X_{24}^{3}$ & $(1,5,13,17)$ & $\left((-1)^{12},0^2\right)$ \\ \hline
						$X_{24}^{4}$ & $(1,7,13,19)$ & $\left((-1)^5,1^2,-1,0^2,-1,0\right)$ \\ \hline
						$X_{24}^{5}$ & $(1,11,17,19)$ & $\left((-1)^3,0,-1,1^3,-1, 0^3\right)$ \\ \hline \hline
						
						$X_{11}^{2}$ & $(1,2,3,4,6)$ & $\left((-1)^2,1,0,1\right)$ \\ \hline
						$X_{11}^{3}$ & $(1,2,3,5,7)$ & $\left((-1)^3,1,-1\right)$ \\ \hline
						$X_{11}^{4}$ & $(1,3,4,5,9)$ & $\left(1,0^2,1^2\right)$ \\ \hline
					\end{tabular}			\label{ptable7}
					\end{tabular}}
					\normalsize
				\end{center}
Here we used the shorthand notation $a^k = \underbrace{(a,...,a)}_{k \text{ times}}$ for an integer $a$.
				
				Hence we have proved
				
				\begin{cor}
					Let $X = (B_1 \times B_2)/(G \times T)$ be a BdF-manifold. Then the family containing $X$ contains a BdF-variety.
				\end{cor}
				
				This concludes the proof of Theorem \ref{class-result}.
				
				\begin{rem}
					Quite interesting is also the datum of explicit equations for the given variety, or one isogenous to it. This problem requires a lot of time and effort and is therefore not dealt with in this paper. See for instance \cite{Catanese-Ciliberto} for results in this direction.
				\end{rem}

			\end{document}